\documentclass[11pt,fleqn,a4paper]{article}
\usepackage{amsmath}
\usepackage{enumerate}
\usepackage{amsfonts}
\usepackage{amsthm}
\usepackage{graphicx}
\usepackage[left=2cm,right=1.5cm,top=2cm,bottom=2cm]{geometry}
\newtheorem{twr}{Theorem}
\newtheorem{dfn}{Definition}
\newtheorem{prop}{Proposition}

\usepackage{authblk}

\title{\textbf{Generalized symmetry reduction\\ of nonlinear differential equations}}
\author[1,2]{I.M.~Tsyfra}
\author[3]{W.~Rzeszut}
\author[1]{V.A.~Vladimirov}
\affil[1]{AGH University of Science and Technology, Faculty of Applied Mathematics, 30 Mickiewicza Avenue, 30-059, Krakow, Poland}
\affil[2]{Institute of Geophysics of NAS of Ukraine, Kyiv, Ukraine}
\affil[3]{Faculty of Mathematics and Computer Science, Jagiellonian University, Krakow, Poland}

\date{}

\begin{document}

\maketitle

\begin{abstract}
We study the application of generalized symmetry for reducing nonlinear partial differential equations. We construct the ansatzes for dependent variable $u$ which reduce the scalar partial differential equation with two independent variables to systems of ordinary differential equations. The operators of Lie--B\"acklund symmetry of the second order ordinary differential equation are used. We apply the method to nonlinear evolutionary equations and find solutions which cannot be obtained in the framework of classical Lie approach. The method is also applicable to partial differential equations which are not restricted to evolution type ones. We construct the solution of nonlinear hyperbolic equation depending on an arbitrary smooth function on one variable. We study also the correlation between the dimension of symmetry Lie algebra and possibility of constructing non-invariant solutions to the equation under study.
\end{abstract}

\section{Introduction}

It is well known that the conditional Lie--B\"acklund symmetry method \cite{FokLiu, Zhd} is successfully used for constructing exact solutions of nonlinear evolution equations. In \cite{Zhd} Zhdanov proved the theorem on the connection between the generalised conditional symmetry and reduction of evolutionary equations to the system of ordinary differential equations. It is worth pointing out that the number of differential equations in this system is equal to the number of unknown functions.

It turns out that
the symmetry reduction method is also applicable to nonevolutionary differential equation
\[\eta(t,x,u,u_x,\ldots , u_{(k)})=0\]
if $X=\eta(t,x,u,u_x,\ldots , u_{(k)})\partial_u$, where $u_{(k)}$ denotes all partial derivatives with respect to $t$, $x$ of $k$-th order, is the Lie--B\"acklund symmetry operator of a nonlinear or linear ordinary differential equation \cite{Tsy}. Thus the approach can be used for symmetry reduction of partial differential equations which are not restricted to evolution type ones. It can be applied to nonevolutionary equations and even to ordinary differential equations.
Moreover in the framework of our approach one can construct an ansatz which reduces nonevolutionary partial differential equations to the system of ordinary differential equations and also the number of equations is smaller then the number of unknown functions. This enables us to find solutions depending on arbitrary functions.
In this paper we use the method proposed in \cite{Tsy}, which is a~generalization of the Svirshchevskii's method~\cite{Svirsh}, meaning we can analyze symmetries of a nonlinear (or nonhomogeneous linear) equation together with ODEs which include, besides dependent and independent variables, parametric variables and derivatives with respect to them. We apply the method to nonlinear diffusion equations in nonhomogeneous medium which enables us to reduce it to the system of two ODEs and obtain exact solutions of the equation under study.

\section{On application of the Lie--B\"acklund symmetry reduction method\\ to differential equations with two independent variables}
In this section we discuss the connection between symmetry and reduction of partial differential equations.
At first we give the definition of Lie--B\"acklund symmetry operator of ordinary differential equation.
Consider differential equation
\begin{equation}\label{orddef}
H\left (t,x,u, \frac{\partial u}{\partial x},\ldots, \frac{\partial^p u}{\partial x^p} \right )=0,
\end{equation}
where $u=u(t,x)$, $H$ is a smooth function. Let us denote by $L$ a set of all differential consequences with respect to $t$ and $x$. We treat \eqref{orddef} to be a partial differential equation which does not include the partial derivatives with respect to $t$.

\begin{dfn} We say that
\begin{equation*}
X=\eta(t,x,u,u_x,\ldots , u_{(k)})\partial_u
\end{equation*}
is a Lie--B\"acklund symmetry operator of equation \eqref{orddef} if the following condition
\begin{equation*}
X^{(k)} H\left (t,x,u, \frac{\partial u}{\partial x},\ldots, \frac{\partial^p u}{\partial x^p} \right ) \Big|_{L}=0
\end{equation*}
holds, where $X^{(k)}$ is a prolongation of the generator of a Lie--B\"{a}cklund symmetry
$X$.
\end{dfn}
At first we consider differential equations obtained with the help of symmetry operators of equation
\begin{equation}\label{2dest}
u_{xx}-\frac{2}{x^2}u=0,
\end{equation}
 where $p(x)=\frac{2}{x^2}$ is the solution of stationary Korteweg--de~Vries equation $p_{xxx}-6pp_x=0$. This equation
admits the $3$-dimensional Lie algebra with basic operators $Q_1=u_t\partial_u$, $Q_2=\big(u_{xxx}-3\frac{u_{xx}u}{u_x}\big)\partial_u$, $Q_3=u\partial_u$.
 Equation \eqref{2dest} is integrable by quadratures and thus we obtain the ansatz
\begin{equation}\label{2deanz}
u=\varphi_1(t)x^2+\frac{\varphi_2(t)}{x},
\end{equation}
where $\varphi_1(t)$, $\varphi_2(t)$ are unknown functions, which reduces nonlinear evolution equation
\begin{equation}
u_t-u_{xxx}+3\frac{u_{xx}u_x}{u}-\lambda u=0, \label{tsyfra:19a}
\end{equation}
where $\lambda={\rm const}$ to the system of two ordinary differential equations
\[
\varphi_1'(t)=\lambda\varphi_1, \quad \varphi_2'(t)=\lambda\varphi_2-12\varphi_1.
\]
The solution of this system has the form
\begin{equation}\label{solre}
\varphi_1=C_1{\rm e}^{\lambda t},\quad \varphi_2=(C_2-12C_1t){\rm e}^{\lambda t},\quad C_1, C_2={\rm const}.
\end{equation}
Substituting \eqref{solre} in \eqref{2deanz} one can obtain the solution of \eqref{tsyfra:19a}.

Next we consider solution $p(t,x)=\frac{x}{6(t+1)}$ of nonstationary Korteweg--de~Vries equation $p_t =p_{xxx}-6pp_x$. In this case the ordinary differential equation
\begin{equation}\label{airye}
u_{xx}-\frac{x}{6(t+1)}u=0
\end{equation}
does not admit the operators $Q_1=u_t\partial_u$ and $Q_2=\big(u_{xxx}-3\frac{u_{xx}u}{u_x}\big)\partial_u$ but admits the operator $Q_1-Q_2=Q_3=\big(u_t-u_{xxx}+3\frac{u_{xx}u}{u_x}\big)\partial_u$ and $Q_4=u\partial_u$.
Equation \eqref{airye} is not integrable by quadratures and its solution is expressed by Airy's functions.
One can now apply the reduction method again with \eqref{2dest} replaced by \eqref{airye} to obtain the ansatz which explicitly depends on $t$ and reduces the evolution equation \eqref{tsyfra:19a} to a~system of ordinary differential equations.
Note that operators $Q_3$, $Q_4$ admitted by equation \eqref{airye} form the 2-dimensional Lie algebra with commutation relation
$[Q_3,Q_4]=Q_3$. Nevertheless equation \eqref{airye} is not integrable by quadratures. This observation seems to be of independent interest.

Our next goal is to show that the method is applicable to nonevolution type partial differential equations. For this purpose we modify equation \eqref{airye} by adding the term $\frac{5}{16x^2}$ which vanishes as $x\to \infty$ and therefore we will consider the following differential equation
\begin{equation}\label{airye8}
u_{xx}- \left (\frac{x}{6(t+1)}+\frac{5}{16x^2}\right )u=0.
\end{equation}
It is invariant with respect to the two-parameter Lie group of point transformations with generators
\[
Y_0=u\partial_u, \quad Y_1=\left (\frac{2}{\sqrt{x}}u_x+\frac{1}{2}x^{-3/2}u \right )\partial_u=\eta_1\partial_u.
\]
From this it immediately follows that equation \eqref{airye8} is integrable by quadratures. To obtain nonevolutionary equations we look for symmetry operators of the form $Y=\eta(t,x,u,u_x,u_t,u_{xt})\partial_u$.
We proved that it admits the Lie--B\"acklund operators
\[
Y_2=\left (\frac{2}{\sqrt{x}}u_{xt}+\frac{1}{2}x^{-3/2}u_t+\frac{x^{3/2}}{9(t+1)^2}u \right )\partial_u=\eta_2 \partial_u,
\quad Y_3=\sqrt{x}\left ( \frac{3(t+1)}{2}\eta_1^2-y^2 \right )y \partial_u=\eta_3 \partial_u
\]
as well. We emphasize that $\eta_2$ depends explicitly on parametric variable $t$ and derivatives $u_{tx}$. This enables us to find nonevolutionary nonlinear differential equations which can be reduced to a system of ordinary differential equations by appropriate ansatz. One can use a linear combination or commutators of operators $Y_1$, $Y_2$, $Y_3$. We consider the symmetry operator $Y_2-Y_3$. Then nonevolutionary differential equation is written in the form
\begin{equation}\label{nevequat}
\eta_2=\eta_3
\end{equation}
in this case.
 By using symmetry properties we have integrated equation \eqref{airye8} and thus obtained the ansatz
\begin{equation*}
u=x^{-1/4}\left (\varphi_1(t){\rm e}^{k(t)\omega}+\varphi_2(t){\rm e}^{-k(t)\omega} \right ),
\end{equation*}
where $\omega =x^{3/2}$, $k(t)=\frac{\sqrt{2}}{3\sqrt{3(t+1)}}$, $\varphi_1(t)$, $\varphi_2(t)$ are unknown functions,
which reduces \eqref{nevequat} to the system of ordinary differential equations
\begin{equation*}
\varphi_1'-\frac{\varphi_1}{2(t+1)}=-2\sqrt{6(t+1)}\varphi_1^2\varphi_2, \quad
\varphi_2'-\frac{\varphi_2}{2(t+1)}=2\sqrt{6(t+1)}\varphi_1\varphi_2^2.
\end{equation*}
This example demonstrates rather strikingly that the method can be applied to nonevolution type partial differential equation. Moreover applying of this method to nonevolutionary nonlinear hyperbolic equation
\begin{equation}\label{waveeqat}
u_{x_1x_2}=u_{x_1}F(u_{x_1}-u),
\end{equation}
where $F$ is arbitrary function on one variable yields the solution
\begin{equation*}
u=-\varphi_1(x_2)+ C{\rm e}^{x_1+\int F(\varphi_1(x_2)){\rm d}x_2}, \quad C={\rm const}
\end{equation*}
parametrized by arbitrary function $\varphi_1(x_2)$. This solution is obtained by using the Lie--B\"acklund symmetry operators
$Q_1=u_{x_1x_2}\partial_u$ and $Q_2=u_{x_1}F(u_{x_1}-u)\partial_u$ of ordinary differential equation $u_{x_1x_1}-u_{x_1}=0$.
The ODE clearly has not to be linear. For example $u_{x_1x_1}+u_{x_1x_1}^2=0$ admits the Lie--B\"acklund operators
$Q_1=\frac{u_{x_1x_2}}{u_{x_1^2}}\partial_u$, and $Q_2=F(u+\ln u_{x_1})\partial_u$, where $F$ is arbitrary function on its argument. One can construct the solution
\begin{equation*}
u=\ln \left (x_1-\int F(\varphi_1(x_2)){\rm d}x_2\right ) +\varphi_1(x_2),
\end{equation*}
where $\varphi_1(x_2)$ is arbitrary smooth function of nonlinear equation
\begin{equation}\label{waveeqa1}
u_{x_1x_2}=u_{x_1}^2F(u+\ln u_{x_1}),
\end{equation}
by virtue of this method.
Obviously we can use Lie--B\"acklund symmetry operators of first order ordinary differential equations too.
For example let consider differential equation
\begin{equation}\label{de1ord}
\psi_x=u\psi^2+v.
\end{equation}
It turns out that equation \eqref{de1ord} admits the Lie--B\"acklund operator
\[Q=\left(-\psi_t+\exp\left(\frac{\psi_x-v}{\psi^2}\right)\cdot\left(\frac{\psi_x-v}{\psi^2}\right)_x\psi^2+\beta\right)\partial_\psi\]
provided $u$, $v$ satisfy the following system of defining equations
\begin{equation}\label{ordsys1}
u_t=({\rm e}^uu_x)_x,\quad v_t=\left (\frac{{\rm e}^uu_xv}{u} \right )_x.
\end{equation}
From what has already been shown we conclude that the reduction method can be applied to nonlinear differential equation
\[\psi_t=\exp\left(\frac{\psi_x-v}{\psi^2}\right)\cdot\left(\frac{\psi_x-v}{\psi^2}\right)_x\psi^2+\frac{{\rm e}^uu_xv}{u}, \]
where $u(t,x)$, $v(t,x)$ is a solution of system \eqref{ordsys1}.

The main idea of the paper is illustrated by the example of the Korteweg--de~Vries equation. One can verify that equation
\begin{equation}
u_{xx}+\frac{u^2}{2}= 0.
\label{zwy1d}
\end{equation}
is invariant with respect to Lie--B\"acklund vector field $X=(u_{xxx}+uu_x)\frac{\partial}{\partial u}$.
Note that we can obtain any desired coefficients for the terms of equation $u_t=u_{xxx}-6uu_x$ by rescaling dependent variable and independent variables $t$ and $x$.
We have proved that equation \eqref{zwy1d} is invariant with respect to a two parameter group of contact transformations.
The basis elements of corresponding Lie algebra are as follows
\begin{gather*}
Q_1=u_xh_1\big(3u_x^2+u^3\big)\frac{\partial}{\partial u},
\\
Q_2=u_x h_2\big(3u_x^2+u^3\big)\int^{u}_0 \frac{{\rm d} s}{\big(u_x^2+\frac{u^3}{3}-\frac{s^3}{3}\big)^{\frac{3}{2}}}\frac{\partial}{\partial u},
\end{gather*}
where $h_1$, $h_2$ are arbitrary smooth functions.
The advantage of using the approach lies in the fact that a~linear combination $\alpha_1X+\alpha_2Q_1+\alpha_3Q_2$ where
$\alpha_1$, $\alpha_2$, $\alpha_3$ are arbitrary real constants is also the symmetry operator of \eqref{zwy1d} and therefore
 the method can applied to nonlinear differential equation
\begin{equation}\label{kdv4}
u_t=u_{xxx}+uu_x+u_xh_1\big(3u_x^2+u^3\big)+u_x h_2\big(3u_x^2+u^3\big)\int^{u}_0 \frac{{\rm d} s}{\big(u_x^2+\frac{u^3}{3}-\frac{s^3}{3}\big)^{\frac{3}{2}}}.
\end{equation}
Then the anzatz
\begin{equation}\label{kdv3}
\int^{u}_0 \frac{{\rm d} s}{\sqrt{\varphi_1(t)-\frac{s^3}{3}}}=x+\varphi_2(t)
\end{equation}
generated by \eqref{zwy1d} reduces partial differential equation \eqref{kdv4}
to the system of two ordinary differential equations
\begin{equation*}
\varphi_1'(t)=2h_2(3\varphi_1(t)),\quad
\varphi_2'(t)=h_1(3\varphi_1(t)).
\end{equation*}
Equation \eqref{kdv4} in the general form is invariant under 2-dimensional Lie algebra with basis elements
$\{\partial_t, \partial_x \}$. Taking into account that $\frac{\partial u}{\partial \varphi_1}$, $\frac{\partial u}{\partial \varphi_2}$ obtained from \eqref{kdv3} are linearly independent one can easily see that the solutions of \eqref{kdv4}
in general case are not invariant with respect to one parameter Lie group with generator $\alpha_1\partial_t+\alpha_2\partial_x$
where $\alpha_1^2+\alpha_2^2\not = 0$. Hence we conclude that the method enables us to construct solutions to equations from class \eqref{kdv4} which are not invariant in the classical Lie sense.

 Next we consider a nonlinear evolutionary equation that describes diffusion processes in inhomogeneous medium.
We are looking for second order ordinary differential equations of the form
\begin{equation}\label{odedfz}
u_{xx}=V(x,u,u_x),
\end{equation}
 which belongs to the class
\eqref{orddef} when $\frac{\partial H}{\partial t}=0$, 
admitting a Lie--B\"acklund symmetry with a generator of the form $X=\big(\frac{K(x)}{u}\big)_{xx}\frac{\partial}{\partial u}$ corresponding to the right hand side of our diffusion equation. We assume $K$ and~$V$ are some sufficiently smooth functions and~$K$ is nonzero. Function~$V$ should satisfy the determining equation
\[X^{(2)}\big(u_{xx}-V(x,u,u_x)\big)\Big|_{M}=0\]
for a given function $K$, where $X^{(2)}$ is a prolongation of the generator of a Lie--B\"{a}cklund symmetry
$X$, and $M$ is a set of all differential consequences of equation \eqref{odedfz} with respect to the variable $x$. Note that condition is equivalent to that for the generalized conditional symmetry. Therefore one can use the generalized conditional symmetry method for searching for functions $V$ as was shown in~\cite{Srg}.

 For the sake of being able to split the equation above into an overdetermined system of differential equations, we restrict our search to a function $V=\sum_{i,j\in\mathbb{Z}}A_{ij}(x)u^iu_x^j$, that is a power series in both $u$ and $u_x$. We focus not a complete classification but rather particular $K$ and $V$ for which this method may produce nonclassical solutions. Due to a great number of determining equations and cases to consider we will move past the explanation of calculations and present the results in the following Proposition.
\begin{prop}

Equation $u_{xx}=V(x,u,u_x)$ admits the LBS operator
$X=\big(\frac{K(x)}{u}\big)_{xx}\partial_u$
if K(x) and V have the form
\begin{enumerate}[$a)$]
\item $K(x)=\mathrm{e}^{\beta x}$, $V=3\frac{u_{x}^2}{u}-3\beta u_{x}+\beta^2u$,
\item $K(x)=(x+\gamma)^{\alpha}$, $V=3\frac{u_{x}^2}{u}+\frac{3-3\alpha}{x+\gamma}u_{x}+\frac{(\alpha-2)(\alpha-1)}{(x+\gamma)^2}u$,
\item $K(x)=(x+\gamma)^{\alpha}$, $V=3\frac{u_{x}^2}{u}+\frac{1-3\alpha}{x+\gamma}u_{x}+\frac{(\alpha-2)\alpha}{(x+\gamma)^2}u$,
\item $K(x)=1$, $V=3\frac{u_{x}^2}{u}+k_1u_{x}u+k_2u^3+k_3xu^4+k_4u^4$,
\item $K(x)=1$, $V=3\frac{u_{x}^2}{u}+\frac{k_2}{k_2x+k_3}u_{x}+k_1(uu_{x}+\frac{k_2}{k_2x+k_3}u^2)+\frac{k_4}{k_2x+k_3}u^3$,
\item $K(x)=(x+\gamma)^{-\frac{1}{2}}$, $V=3\frac{u_{x}^2}{u}+\frac{5}{2(x+\gamma)}u_{x}+\frac{5}{4(x+\gamma)^2}u+k_1(x+\gamma)^{\frac{5}{2}}u^4+k_2(x+\gamma)^{\frac{1}{2}}u^4$,
\item $K(x)=(x+\gamma)^{-1}$, $V=3\frac{u_{x}^2}{u}+\frac{2(3k_2(x+\gamma)^2+2k_1)}{(x+\gamma)(k_2(x+\gamma)^2+k_1)}u_{x}+\frac{3(2k_2(x+\gamma)^2+k_1)}{(k_2(x+\gamma)^2+k_1)(x+\gamma)^2}u+\frac{k_3(x+\gamma)^2u^3}{k_2(x+\gamma)^2+k_1}$, $k_1\neq0 \vee k_2\neq 0$,
\item $K(x)=(x+\gamma)^{-\frac{3}{2}}$, $V=3\frac{u_{x}^2}{u}+\frac{15}{2(x+\gamma)}u_{x}+\frac{35}{4(x+\gamma)^2}u+k_1(x+\gamma)^{\frac{9}{2}}u^4+k_2(x+\gamma)^{\frac{5}{2}}u^4$,
\item $K(x)=(x+\gamma)^{-2}$, $V=3\frac{u_{x}^2}{u}+\frac{10k_2(x+\gamma)+9k_3}{(x+\gamma)(k_2(x+\gamma)+k_3)}u_{x}+\frac{3(5k_2(x+\gamma)+4k_3)}{(x+\gamma)^2(k_2(x+\gamma)+k_3)}u$\\
$+k_1\big((x+\gamma)uu_{x}+\frac{3k_2(x+\gamma)+2k_3}{(k_2(x+\gamma)+k_3)}u^2\big)+\frac{k_4(x+\gamma)^3}{k_2(x+\gamma)+c3}u^3$, $k_2\neq0 \vee k_3\neq 0$,
\item $K(x)=(x+\gamma)^{-2}$, $V=3\frac{u_{x}^2}{u}+\frac{10}{x+\gamma}u_{x}+\frac{15}{(x+\gamma)^2}u+k_1\big((x+\gamma)uu_{x}+3u^2\big)+k_2(x+\gamma)^{5}u^4$\\
$+k_3(x+\gamma)^{4}u^4+k_4(x+\gamma)^{2}u^3$,
\item $K(x)=(x+\gamma)^{-2}$, $V=3\frac{u_{x}^2}{u}+\frac{7k_1+10k_2(x+\gamma)}{(x+\gamma)(k_1+k_2(x+\gamma))}u_{x}+\frac{15k_2^2(x+\gamma)^2+21k_1k_2(x+\gamma)+8k_1^2}{(x+\gamma)^2(k_2(x+\gamma)+k_1)^2}u$, $k_1\neq0 \vee k_2\neq 0$,
\end{enumerate}
where $\alpha$, $\beta$, $\gamma$, $k_i$, $i=1,2,3,4$, are all constants.
\end{prop}
It must be noted that because of the said restrictions this is not a complete classification.

Solutions of ordinary differential equations is used as ansatzes which produce a reduction of the equation $u_t=(\frac{K}{u})_{xx}$ to a system of two ODEs. We use the fact that the equation{\samepage
\[u_t=\left(\frac{K}{u}\right)_{xx}+\eta(x,u,u_x)\]
 also allows reduction with the same ansatz, for $\eta$ such that $\eta\frac{\partial}{\partial u}$ is the symmetry of the ODE.}

We have considered separate cases depending on the heterogeneity of the medium $K(x)$ (homogeneity for $K={\rm const}$) and the ODE (which is not unique for the choice of $K$), calculated the full contact symmetry and the reduced equations. From now on $A_i$ are arbitrary smooth functions on their two arguments.
\begin{enumerate}[(i)]
\item When $K(x)={\rm e}^{\beta x}$ and
\begin{equation}
u_{xx}=3\frac{u_x^2}{u} -3\beta u_x +\beta^2 u
\label{beta2}
\end{equation}
 we obtain partial differential equation
\begin{equation}
u_t=\left(\frac{{\rm e}^{\beta x}}{u}\right)_{xx}+
\sum_{i=1}^2
{{\rm e}^{-i\beta x}}u^3
A_i
\left(
2\frac{{\rm e}^{\beta x}}{u^3}\Big(u-\frac{u_x}{\beta}\Big)
,
\frac{{\rm e}^{2\beta x}}{u^3}\Big(2\frac{u_x}{\beta}-u\Big)
\right).
\label{ubeta}
\end{equation}
From what already has been shown if follows that the ansatz
\begin{equation}
\label{an}
u(x,t)=\frac{\pm {\rm e}^{\beta x}}{\sqrt{\varphi_1(t){\rm e}^{\beta x}+\varphi_2(t)}},
\end{equation}
which is general solution of \eqref{beta2},
reduces equation \eqref{ubeta} to the system of ordinary differential equations
\begin{gather*}
\varphi_1'+\frac{\beta^2}{2}\varphi_1^2+2A_1(\varphi_1,\varphi_2)=0,\nonumber\\
\varphi_2'+\beta^2\varphi_1\varphi_2+2A_2(\varphi_1,\varphi_2)=0,\label{ree}
\end{gather*}

 \item when $K(x)=x^{\alpha},~\alpha\neq 0$, and $u_{xx}=3\frac{u_x^2}{u} +\frac{3-3\alpha}{x} u_x+ \frac{(\alpha-2)(\alpha-1)}{x^2}u$ we obtain, in a similar way, the equation
\begin{equation}
u_t=\left(\frac{x^{\alpha}}{u}\right)_{xx}+
\sum_{i=1}^2
x^{2-i\alpha}u^3
A_i
\left(
\frac{2x^{\alpha-2}}{\alpha u^3}\big((\alpha-1)u-xu_x\big)
,
\frac{x^{2\alpha-2}}{\alpha u^3}\big((2-\alpha)u+2xu_x\big)
\right)
\label{xa}
\end{equation}
and the ansatz
\begin{equation}
u(x,t)=\frac{\pm x^{\alpha-1}}{\sqrt{\varphi_1(t)x^{\alpha}+\varphi_2(t)}}.
\label{xaa}
\end{equation}
Substituting \eqref{xa} into \eqref{xaa} yields the reduced system
\begin{gather*} \varphi_1'+\frac{\alpha(\alpha+1)}{2}\varphi_1^2+2A_1(\varphi_1,\varphi_2)=0,\\
 \varphi_2'+\alpha(\alpha+1)\varphi_1\varphi_2+2A_2(\varphi_1,\varphi_2)=0.
 \end{gather*}

Analysis similar to that in cases (i) and (ii) gives the following results:
 \item when $H(x)=x^{\alpha}$, $\alpha\neq -2$ and $u_{xx}=3\frac{u_x^2}{u} +\frac{1-3\alpha}{x}u_x +\frac{(\alpha-2)\alpha}{x^2}u$, the equation
\begin{gather*}
u_t=\left(\frac{x^{\alpha}}{u}\right)_{xx}+ \frac{u^3}{x^{\alpha-2}} A_1\left(\frac{2x^{\alpha-2}}{(a+2)u^3}\big(\alpha u-xu_x\big),\frac{x^{2\alpha}}{(a+2)u^3}\big((2-\alpha)u+2xu_x\big)\right)
\\
\hphantom{u_t=}{} +
\frac{u^3}{x^{2\alpha}}A_2\left(\frac{2x^{\alpha-2}}{(\alpha+2)u^3}\big(\alpha u-xu_x\big),\frac{x^{2\alpha}}{(\alpha+2)u^3}\big((2-\alpha)u+2xu_x\big)\right),
\\
u_t=\left(\frac{x^{\alpha}}{u}\right)_{xx}+
\sum_{i=1}^2
x^{4-i(\alpha+2)}u^3
A_i
\left(
\frac{2x^{\alpha-2}}{(\alpha+2)u^3}\big(\alpha u-xu_x\big)
,
\frac{x^{2\alpha}}{(\alpha+2)u^3}\big((2-\alpha)u+2xu_x\big)
\right)
\end{gather*}
admits the ansatz
\[u(x,t)=\frac{\pm x^{\alpha}}{\sqrt{\varphi_1(t)x^{\alpha+2}+\varphi_2(t)}},\]
for which it is reduced to the system
\begin{gather*} \varphi_1'+\frac{\alpha(\alpha+1)}{2}\varphi_1^2+2A_1(\varphi_1,\varphi_2)=0,\\
\varphi_2'+(\alpha+1)(\alpha+2)\varphi_1\varphi_2+2A_2(\varphi_1,\varphi_2)=0,\end{gather*}

 \item when $K(x)=x^{-2}$ and $u_{xx}=3\frac{u_x^2}{u} +\frac{7}{x} u_x+ \frac{8}{x^2}u$,
\begin{gather*}
u_t=\big(\frac{1}{x^2u}\big)_{xx}+
x^4\ln(x)u^3
A_1
\left(
\frac{-2(xu_x+2u)}{x^4u^3}
,
\frac{2x\ln(x) u_x+(4\ln(x)+1)u}{x^4u^3}
\right)\\
\hphantom{u_t=}{} +
x^4u^3
A_2
\left(
\frac{-2(xu_x+2u)}{x^4u^3}
,
\frac{2x\ln(x) u_x+(4\ln(x)+1)u}{x^4u^3}
\right)
\end{gather*}
the equation
\[
u_t=\left(\frac{1}{x^2u}\right)_{xx}+
\sum_{i=1}^2
x^4\ln(x)^{2-i}u^3
A_i
\left(
-2\frac{xu_x+2u}{x^4u^3}
,
\frac{2x\ln x u_x+(4\ln x+1)u}{x^4u^3}
\right)
\]
admits the ansatz
\[u(x,t)=\frac{\pm 1}{x^2\sqrt{\varphi_1(t)\ln(x)+\varphi_2(t)}},\]
for which it is reduced to the system
\begin{gather*} \varphi_1'-\varphi_1^2+2A_1(\varphi_1,\varphi_2)=0,\\
 \varphi_2'-\tfrac{1}{2}\varphi_1^2-\varphi_1\varphi_2+2A_2(\varphi_1,\varphi_2)=0,\end{gather*}
 \item when $K(x)=x^{-2}$ and $u_{xx}=3\frac{u_x^2}{u} +\frac{10}{x} u_x+ \frac{15}{x^2}u$,
\begin{gather*}
u_t=\left(\frac{1}{x^2u}\right)_{xx}+
x^6u^3
A_1
\left(
-\frac{5u+2xu_x}{x^6u^3}
,
2\frac{3u+xu_x}{x^5u^3}
\right)\\
\hphantom{u_t=}{}
+
x^5u^3
A_2
\left(
-\frac{5u+2xu_x}{x^6u^3}
,
2\frac{3u+xu_x}{x^5u^3}
\right)
\end{gather*}
the equation
\[
u_t=\left(\frac{1}{x^2u}\right)_{xx}+
\sum_{i=1}^2
x^{7-i}u^3
A_i
\left(
-\frac{5u+2xu_x}{x^6u^3}
,
2\frac{3u+xu_x}{x^5u^3}
\right)
\]
admits the ansatz
\[u(x,t)=\frac{\pm 1}{x^2\sqrt{\varphi_1(t)x^2+\varphi_2(t)x}},\]
for which it is reduced to the system
\begin{gather*} \varphi_1'-\tfrac{1}{2} \varphi_1^2+ 2 A_1(\varphi_1,\varphi_2)=0,\\
 \varphi_2'+2A_2(\varphi_1,\varphi_2)=0.\end{gather*}
In case of homogeneous medium we have:

 \item $K(x)=k={\rm const}$ and $u_{xx}=3\frac{u_x^2}{u} +\frac{3}{x} u_x+\frac{2}{x^2} u$,
\begin{gather*}
u_t=\left(\frac{k}{u}\right)_{xx}+
x^2\ln(x)u^3
A_1
\left(
-2\frac{u+xu_x}{x^2u^3}
,
\frac{(2\ln(x)+1)u+2x\ln(x)u_x}{x^2u^3}
\right)\\
\hphantom{u_t=}{} +
x^2u^3
A_2
\left(
-2\frac{u+xu_x}{x^2u^3}
,
\frac{(2\ln(x)+1)u+2x\ln(x)u_x}{x^2u^3}
\right)
\end{gather*}
the equation
\[
u_t=\left(\frac{k}{u}\right)_{xx}+
\sum_{i=1}^2
x^2(\ln x )^{2-i}u^3
A_i
\left(
-2\frac{u+xu_x}{x^2u^3}
,
\frac{(2\ln x+1)u+2x\ln x u_x}{x^2u^3}
\right)
\]
admits the ansatz
\[u(x,t)=\frac{\pm 1}{x\sqrt{\varphi_1(t)\ln x+\varphi_2(t)}},\]
for which it is reduced to the system
\begin{gather*} \varphi_1'+k\varphi_1^2+ 2 A_1(\varphi_1,\varphi_2)=0,\\
 \varphi_2'-\frac{k}{2}\varphi_1^2+h\varphi_1\varphi_2+2A_2(\varphi_1,\varphi_2)=0.\end{gather*}
\end{enumerate}
For arbitrary $A_1$, $A_2$ the modified equations have only one symmetry operator $\partial_t$, therefore whenever $\varphi_1'\neq0$ or $\varphi_2'\neq0$, the solution will not be an invariant one. Those reduced equations in fact exist, we will show that in the next chapter on the example of point symmetry.
We obtained a class of equations which describe the diffusion process in nonhomogeneous medium, with additional terms for which the method can be applied. In special case, where we have a point symmetry, these additional terms represent nonlinear nonhomogeneous sources and convection processes. In general $A_i$ may depend on parametric variable $t$ and the corresponding diffusion equation may be used for simulation of nonstationary media.

It is obvious that the method is applicable to ordinary differential equation
$(\frac{K}{u})_{xx}+\eta(x,u,u_x)=0$.
The corresponding ansatz reduces ordinary differential equation to system of algebraic (not differential) equations in this case. For example ansatz \eqref{an} reduces ordinary differential equation
\begin{equation*}
\left(\frac{{\rm e}^{\beta x}}{u}\right)_{xx}+
\sum_{i=1}^2
{{\rm e}^{-i\beta x}}u^3
A_i
\left(
2\frac{{\rm e}^{\beta x}}{u^3}\left(u-\frac{u_x}{\beta}\right)
,
\frac{{\rm e}^{2\beta x}}{u^3}\left(2\frac{u_x}{\beta}-u\right)
\right)=0
\end{equation*}
to the system of two algebraic equations
\begin{gather}
\label{rea}
\frac{\beta^2}{2}\varphi_1^2+2A_1(\varphi_1,\varphi_2)=0,\quad
\beta^2\varphi_1\varphi_2+2A_2(\varphi_1,\varphi_2)=0.
\end{gather}
One can easily choose such $A_1$, $A_2$ that the system \eqref{rea} will not have solutions.
In general, the method ensures the reduction of ordinary differential equations to system of algebraic equations but does not
guarantee the existence of even one solution of reduced system and consequently solution of overdetermined system given by equation under study and ordinary differential equation possessing the corresponding Lie--B\"acklund symmetry.

\section{Construction of solutions of nonlinear diffusion equations\\ via generalized symmetry method}

We modify our original diffusion equation by some selected characteristics of point symmetries representing some physical properties of nonlinearity on nonhomogeneity.
Let's start with the equation
\begin{equation}
u_{xx}=3\frac{u_x^2}{u}-3\beta u_x+\beta^2 u,\quad\beta\neq 0.
\label{zwy3}
\end{equation}
The ansatz is
\begin{equation}
\label{ans3}
u(x,t)=\frac{{\rm e}^{\beta x}}{\sqrt{\varphi_1(t){\rm e}^{\beta x}+\varphi_2(t)}}.
\end{equation}
We will consider equation
\begin{equation}
u_t=\left(\frac{\mathrm{e}^{\beta x}}{u}\right)_{xx}+
a_1u+a_2u_x+a_3u^3\mathrm{e}^{-\beta x}+a_4u^3\mathrm{e}^{-2\beta x}+a_5\big(\beta u\mathrm{e}^{-\beta x}-\mathrm{e}^{-\beta x}u_x\big),\quad a_1,a_2,\dots,a_5={\rm const},
\label{a1a5}
\end{equation}
which is obtained from \eqref{ubeta} by letting
\begin{gather*}
A_1=-a_1I_1-a_2\frac{\beta}{2}I_1+a_3,\\
A_2=-a_1I_2 + a_2\beta I_2+ a_4+ a_5\frac{\beta}{2}I_1,
\end{gather*}
where $I_1=2\frac{{\rm e}^{\beta x}}{u^3}\big(u-\frac{u_x}{\beta}\big)$, $I_2=\frac{{\rm e}^{2\beta x}}{u^3}\big(2\frac{u_x}{\beta}-u\big)$.

Note, that equation \eqref{zwy3} can be linearized using a point change of variables, but we don't do such thing, because it would change also the term $\big(\frac{\mathrm{e}^{\beta x}}{u}\big)_{xx}$. The property of reduction for equation \eqref{a1a5} is not limited to linear ODEs, it is also valid for nonlinear ODEs. For example, Korteweg--de~Vries equation admits reduction with an ansatz obtained from a solution to an ODE which is an arbitrary linear combination of higher-order symmetries. One can also use any linear combination of higher-order symmetries of KdV to produce the ansatz.
 After substituting the ansatz \eqref{ans3} into \eqref{a1a5} we obtain reduction equations:
 \begin{gather} \label{red1}
2\varphi_1'+\beta^2\varphi_1^2+2(2a_1+\beta a_2)\varphi_1+4a_3=0,\\
 \label{red2}
2\varphi_2'+2\beta^2 \varphi_1\varphi_2+4a_1\varphi_2+4\beta a_2 \varphi_2+4 a_4+2\beta a_5\varphi_1=0.
\end{gather}
Now, the evolutionary equation \eqref{a1a5} admits the following symmetry operators:
\begin{gather*}
Y_1=\partial_t,\\
Y_2=\frac{1}{\beta}\partial_x+\frac{u}{2}\partial_u, \text{ if }a_4 = a_5=0,\\
Y_2=t\partial_t+\frac{2}{\beta}\partial_x+\frac{3}{2} u\partial_u, \text{ if }a_1=a_2=a_3=a_5=0,~a_4\neq 0,\\
Y_2=t\partial_t+\frac{1}{\beta}\partial_x+u\partial_u, \text{ if }a_1=a_2=a_3=a_4 = 0,~a_5\neq 0,\\
Y_3=\mathrm{e}^{(a_2\beta+2a_1)t}\partial_t-a_2\mathrm{e}^{(a_2\beta+2a_1)t}\partial_x+a_1u\mathrm{e}^{(a_2\beta+2a_1)t}\partial_u, \text{ if }a_3 = a_4 = a_5=0,~a_2\beta+2a_1\neq0,\\
Y_3=t\partial_t-a_2t\partial_x+\left(a_1t+\frac{1}{2}\right)u\partial_u \text{ if }a_3=a_4=a_5=0,~a_2\beta+2a_1=0 .\end{gather*}
Depending on the choice of $a_i$ the reduced system is integrable and we can construct the solutions to the modified evolutionary equation.

For $u_t=\Big(\frac{\mathrm{e}^{\beta x}}{u}\Big)_{xx}+a_1u+a_2u_x+a_3u^3\mathrm{e}^{-\beta x}+a_4u^3\mathrm{e}^{-2\beta x}+a_5(\beta u\mathrm{e}^{-\beta x}-\mathrm{e}^{-\beta x}u_x)$, $a_2\neq0
$, $\gamma=\pm\sqrt{(a_2\beta+2a_1)^2- 4\beta^2a_3}$, $\gamma\neq 0$, $\delta =\frac{ a_1 a_5 - a_4 \beta + a_2 a_5 \beta}{a_1^2 + a_1 a_2 \beta -
 a_3 \beta^2}$, $a_1^2 + a_1 a_2 \beta -
 a_3 \beta^2\neq 0$ the reduced system \eqref{red1}--\eqref{red2} has the following solution:
\begin{gather*} \varphi_1(t)=\frac{\gamma\tanh\big(\frac{\gamma}{2}(t+s_1)\big)-\beta a_2-2 a_1}{\beta^2},\\
\varphi_2(t)=\frac{s_2\mathrm{e}^{-\beta a_2t}-(2a_5+\beta a_2 \delta)\cosh(\gamma(t+s_1))
+\delta\gamma\sinh(\gamma(t+s_1))+\frac{4}{\beta a_2}(a_1a_5-\beta a_4)+2a_5
}
{4\beta\cosh^2\big(\frac{\gamma}{2}(t+s_1)\big)}.
\end{gather*}
Substituting $\varphi_1$, $\varphi_2$ into \eqref{ans3} gives the solution of \eqref{a1a5} with the restrictions above.

The construction of solutions in the rest of the cases runs as before.

For $u_t=\big(\frac{\mathrm{e}^{\beta x}}{u}\big)_{xx}+a_1u+a_3u^3\mathrm{e}^{-\beta x}+a_4u^3\mathrm{e}^{-2\beta x}+a_5(\beta u\mathrm{e}^{-\beta x}-\mathrm{e}^{-\beta x}u_x)$, $\gamma=\pm2\sqrt{a_1^2-\beta^2a_3}$, $\gamma\neq 0$:
\begin{gather*} \varphi_1(t)=\frac{\gamma}{\beta^2}\tanh(\frac{\gamma}{2}(t+s_1))-\frac{2a_1}{\beta^2},\\
 \varphi_2(t)=\frac{s_2+(a_1a_5-\beta a_4)\big(\gamma(t+s_1)+\sinh(\gamma(t+s_1))\big)}{\beta\gamma\cosh^2\big(\frac{\gamma}{2}(t+s_1)\big)}-\frac{a_5}{\beta};
\end{gather*}
for $u_t=\big(\frac{\mathrm{e}^{\beta x}}{u}\big)_{xx}+a_1u+a_2u_x+\big(\frac{a_2}{2}+\frac{a_1}{\beta}\big)^2u^3\mathrm{e}^{-\beta x}+a_4u^3\mathrm{e}^{-2\beta x}+a_5(\beta u\mathrm{e}^{-\beta x}-\mathrm{e}^{-\beta x}u_x)$, $a_2\neq 0$:
\begin{gather*} \varphi_1(t)=\frac{2}{\beta^2(t+s_1)}-\frac{a_2\beta+2 a_1}{\beta^2},\\
 \varphi_2(t)=\frac{s_2\mathrm{e}^{-\beta a_2 t}+4(a_5(a_2\beta+a_1)-\beta a_4)(1-\beta a_2(t+s_1))}{\beta^4a_2^3(t+s_1)^2}+\frac{a_5(a_2\beta+2a_1)-2\beta a_4}{\beta^2 a_2};
 \end{gather*}
for $u_t=\big(\frac{\mathrm{e}^{\beta x}}{u}\big)_{xx}+a_1u+\frac{a_1^2}{\beta^2}u^3\mathrm{e}^{-\beta x}+a_4u^3\mathrm{e}^{-2\beta x}+a_5(\beta u\mathrm{e}^{-\beta x}-\mathrm{e}^{-\beta x}u_x)$:
\begin{gather*} \varphi_1(t)=\frac{2}{\beta^2(t+s_1)}-\frac{2 a_1}{\beta^2},\\
 \varphi_2(t)=\frac{s_2}{(t+s_1)^2}+\frac{2(a_1a_5-a_4)}{3\beta}(t+s_1)-\frac{a_5}{\beta};
 \end{gather*}
and for $u_t=\big(\frac{\mathrm{e}^{\beta x}}{u}\big)_{xx}-\beta a_2u+a_2u_x+a_4u^3\mathrm{e}^{-2\beta x}+a_5(\beta u\mathrm{e}^{-\beta x}-\mathrm{e}^{-\beta x}u_x)$, $a_2\neq0$:
\begin{gather*}\varphi_1(t)=\frac{2a_2}{1+\mathrm{e}^{-a_2\beta(t+s_1)}},\\
\varphi_2(t)=-\frac{
2(2a_4+a_2a_5)\mathrm{e}^{a_2\beta(t+s_1)}+(a_4+a_2a_5)\mathrm{e}^{2a_2\beta(t+s_1)}+2 a_2a_4t+s_2
}{a_2(1+\mathrm{e}^{a_2\beta(t+s_1)})^2}.
\end{gather*}
When $\exists_{i\in\{4,5\}}\colon a_i\neq 0$ and $\exists_{j\in\{1,2,3,4,5\},j\neq i}\colon a_j\neq 0$ the equation \eqref{a1a5} admits only one symmetry opera\-tor,~$\partial_t$. The presented solutions are clearly not invariant under translations of the variable $t$.
When exactly one of the constants $a_4$, $a_5$ is nonzero and $a_1=a_2=a_3=0$ or $a_3$ is nonzero and both~$a_4$ and~$a_5$ are zeros, equation~\eqref{a1a5} admits exactly two symmetry operators and invariance under one-parameter symmetry group with the generator $\alpha_1Y_1+\alpha_2Y_2$ must be verified
 from the definition, that is solution $u=u(x,t)$ is invariant when there exist real numbers $\alpha_1$, $\alpha_2$, at least one nonzero, that $(\alpha_1Y_1+\alpha_2Y_2)(u-u(x,t))\big|_{u=u(x,t)}=0$. Otherwise, the solution is not invariant.
Instead of checking the invariance by the definition, we will compare them with the invariant solutions in the class \eqref{ans3}. Functions $\varphi_i$ for invariant solutions are as follows.

For $a_3\neq 0$, $a_4=a_5=0$
\[\varphi_1=c_1,\quad \varphi_2=c_2~{\rm exp}\left(\frac{\alpha_2}{\alpha_1}t\right),\]
for $a_4\neq 0$, $a_1=a_2=a_3=a_5=0$
\[\varphi_1=\frac{c_1}{\alpha_1+\alpha_2t},\quad \varphi_2=\frac{c_2}{\alpha_1+\alpha_2t},\]
and for $a_5\neq 0$, $a_1=a_2=a_3=a_4=0$
\[\varphi_1=\frac{c_1}{\alpha_1+\alpha_2t},\quad\varphi_2=c_2,\]
where $c_1$, $c_2$, $\alpha_1$, $\alpha_2$ are all constants. By plain comparison, in those 3 cases, none of the five solutions in general form obtained from the reduction equations is invariant under $\alpha_1Y_1+\alpha_2Y_2$ ($Y_2$ depending on the choice of nonzero $a_i$).

From the number of the symmetry operators we can already establish some cases in which the solution will be or not be invariant under the symmetry group. We note that the equation $u_{xx}=3\frac{u_x^2}{u}-3\beta u_x+\beta^2 u$, $\beta\neq 0$ admits the all the symmetry operators equation \eqref{a1a5} does, among others. Whenever $a_3=a_4=a_5=0$, the evolutionary equation~\eqref{a1a5} admits a 3-dimensional symmetry group and since our ansatz solution consists of only 2 constants, the solution must be invariant under that group. On the other hand, when $a_4\neq 0$ or $a_5\neq 0$ the dimension of the symmetry group is smaller than the number of constants in the ansatz (equation admits only one translation operator), therefore the solution will definitely not be invariant. The cases where the dimension of the symmetry group matches the order of the ODE, in that case equal to 2, must be analyzed. In many cases it might be easier to compare invariant solutions arising from the form of the ansatz with the solutions produced via the reduction equations, instead of directly substituting the solutions into invariance equations.

One does not always have to solve the reduction equations to determine if the solution is or isn't invariant. Let's for example take equation
\begin{equation}
\label{a6a7a8}
u_t=\left(\frac{\mathrm{e}^{\beta x}}{u}\right)_{xx}+\mathrm{e}^{\beta x}\left(\frac{\beta}{2}u -u_x\right)\big(a_6+\mathrm{e}^{\beta x}u^{-2}a_7\big) +a_8\left(\frac{\mathrm{e}^{\beta x}}{u}\right)_{x}, \quad a_i={\rm const}, \quad i=6,7,8,
\end{equation}
which is obtained from \eqref{ubeta} by letting
\begin{gather*}
A_1=-a_6\frac{\beta}{2}I_2-a_7\frac{\beta}{2}I_1I_2+a_8\frac{\beta}{2}I_1^2,\\
A_2=-a_7\frac{\beta}{2}I_2^2+a_8\frac{\beta}{2}I_1I_2.
\end{gather*}

It admits the same ansatz as in the previous example,
\begin{equation}\label{ansatzE}
u(x,t)=\frac{{\rm e}^{\beta x}}{\sqrt{\varphi_1(t){\rm e}^{\beta x}+\varphi_2(t)}},
\end{equation}
and is reduced to a system
\begin{gather*} \varphi_1'-\beta a_6\varphi_2-\beta a_7\varphi_1\varphi_2+\beta\big(a_8+\tfrac{1}{2}\beta\big)\varphi_1^2=0,\\
 \varphi_2'-\beta a_7 \varphi_2^2+\beta(a_8+\beta)\varphi_1\varphi_2=0.\end{gather*}
The equation \eqref{a6a7a8} possesses three symmetry operators when $a_6=a_7=0$
and only one symmetry operator ($\partial_t$) when both $a_6\neq 0$ and $a_7\neq 0$.

Let's consider the case $a_6=0$, $a_7\neq 0$ with two-dimensional Lie algebra. Here
$Q_1=\partial_t$, $Q_2=t\partial_t+\frac{1}{2}u\partial_u$ is the basis of the algebra for equation \eqref{a6a7a8} (with $a_6=0$).
Invariance criterion in terms of functions $\varphi_i$ after splitting with respect to the powers of $\mathrm{e}^{\beta x}$ is
\[(\alpha_1+\alpha_2t)\varphi_1'+\alpha_2\varphi_1=0,\quad (\alpha_1+\alpha_2t)\varphi_2'+\alpha_2\varphi_2=0.\]
Reduction equations do have an explicit solution but it is invariant, because $\varphi_1=0$. They also have an implicit solution. At this point all we need is to solve the reduction equations for the derivatives of $\varphi_i$ and substitute those to the system above. The result is another system,
\begin{gather*} \varphi_1\big(\beta(\alpha_1+\alpha_2t)(a_7\varphi_2-a_8\varphi_1-\tfrac{1}{2}\beta\varphi_1)+\alpha_2\big)=0,\\
 \varphi_2\big(\beta(\alpha_1+\alpha_2t)(a_7\varphi_2-a_8\varphi_1-\beta\varphi_1)+\alpha_2\big)=0.\end{gather*}
The two equations are very similar. After dividing the $i$-th equation by $\beta\varphi_1\varphi_i$ and subtracting one from another we have
\[\tfrac{1}{2}(\alpha_1+\alpha_2t)=0,\]
meaning
\[\alpha_1=\alpha_2=0,\]
so the solution \eqref{ansatzE}, where both $\varphi_i$ are nonzero solutions of the reduction equations, wouldn't be an invariant solution.

Now we will find the solution to the system of reduction equation. But firstly, we will show how to find a symmetry of the reduced equations having the symmetry of the partial differential equation. There is a well know theorem for inherit symmetries in a classical case of Lie groups of point transformations. It says that we can obtain a symmetry group of reduced equations if we look for invariant solutions with respect to some subgroup, which is a normalizer of a group $G$ admitted by PDE. When the ansatz is searched for using ODEs of higher order, this condition changes.

We consider a case, when the symmetry operator $Q$ of the PDE is admitted also by the ODE \eqref{zwy3} (we treat the ODE as a PDE that does not include the derivative $u_t$) and condition
\begin{equation}\label{inher1}
Q^{(1)}I_1=F_1(x,u,u_x)=f_1(I_1,I_2), \quad
Q^{(1)}I_2=F_2(x,u,u_x)=f_2(I_1,I_2)
\end{equation}
holds, where $Q^{(1)}$ is the first prolongation of operator $Q$, $f_1(I_1,I_2)$, $f_2(I_1,I_2)$ are arbitrary smooth functions,
\[I_1=\varphi_1={\rm e}^{\beta x}\left\{\frac{2}{u^2}-\frac{2}{\beta}\frac{u_x}{u^3}\right\},\quad
I_2=\varphi_2={\rm e}^{2\beta x}\left\{\frac{2}{\beta}\frac{u_x}{u^3}-\frac{1}{u^2}\right\},\]
 are the first integrals of equation \eqref{zwy3}.
For a independent variable we impose the condition
\begin{equation}\label{inher2}
Qt=m(t),
\end{equation}
where $m(t)$ is arbitrary smooth function. One can construct the symmetry operator for reduced system in the form
 $ \tilde{Q} = m(t)\partial_t+f_1(\varphi_1, \varphi_2)\partial_{\varphi_1}+f_2(\varphi_1, \varphi_2)\partial_{\varphi_2}$ if
 the conditions \eqref{inher1}, \eqref{inher2} are fulfilled. Obviously, we obtain a nontrivial symmetry for a system of reduced equations, if $f_1$, $f_2$ and $m$ are not identically zeros.
It's pretty obvious that $Q_1$ i $Q_2$ satisfy all the conditions, namely for $Q_1$ we have $f_1=0$, $f_2=0$ and $m=1$, and for $Q_2$ we have $f_1=-I_1$, $f_2=-I_2$ and $m=t$.
Thus we conclude that the ODE system
\begin{gather*} \varphi_1'-\beta a_7\varphi_1\varphi_2+\beta\left(a_8+\tfrac{1}{2}\beta\right)\varphi_1^2=0,\\
 \varphi_2'-\beta a_7\varphi_2^2+\beta(a_8+\beta)\varphi_1\varphi_2=0\end{gather*}
is invariant with respect to 2-parameter Lie group of point transformations whose Lie algebra is given by basic elements
 \[X_1=\partial_t,\quad X_2=t\partial_t-\varphi_1\partial_{\varphi_1}-\varphi_2\partial_{\varphi_2},\]
which are obtained from $Q_1$, $Q_2$. It means that the system is solvable in quadratures.
A point transformations
\[T=\frac{\varphi_1}{\varphi_2},\quad W(T)=t,\quad Z(T)=\ln{\varphi_2}\]
maps the symmetry operators $X_1$, $X_2$ into
\[Y_1=\partial_W,\quad Y_2=W\partial_W-\partial_Z.\]
The transformed system can be simplified into
\[W_T=\frac{2}{{\rm e}^Z\beta^2T^2},\quad Z_T=\frac{2}{\beta T^2}\{-(a_8+\beta)T+a_7\}.\]
We can easily solve for $Z$,
\[Z(T)=\ln\left\{c T^{\frac{-2(a_8+\beta)}{\beta}}\exp \left(\frac{-2a_7}{\beta T}\right)\right\},\]
which in initial coordinates is an algebraic equation
\begin{equation}
\varphi_2=c \left(\frac{\varphi_1}{\varphi_2}\right)^{\frac{-2(a_8+\beta)}{\beta}}\exp \left(\frac{-2a_7}{\beta}\frac{\varphi_2}{\varphi_1}\right),\quad c={\rm const}.
\label{wz1}
\end{equation}
Solution
\begin{equation}\label{integral}
W(T)=2\frac{c}{\beta^2}\int T^{\frac{2a_8}{\beta}}\exp\left(\frac{2 a_7}{\beta T}\right){\rm d}T+c_0,\quad c_0, c={\rm const}
\end{equation}
 is trickier to utilize once going back to original coordinates.

If we were to use $W_T=\frac{2}{{\rm e}^Z\beta^2T^2}$, this equation would imply a differential equation
\begin{equation*}
\left(\frac{\varphi_1}{\varphi_2}\right)'=\frac{\beta^2}{2}\frac{\varphi_1^2}{\varphi_2}
\end{equation*}
that can be alternatively written as
\begin{equation}
\left(\frac{\varphi_2}{\varphi_1}\right)'=-\frac{\beta^2}{2}\varphi_2.
\label{wz3}
\end{equation}
The only invariant solution of the form $u(x,t)=\pm\frac{{\rm e}^{\beta x}}{\sqrt{\varphi_1 {\rm e}^{\beta x}+\varphi_2}}$ is the one where $\varphi_i=\frac{c_i}{\alpha_1+\alpha_2t}$, $c_i,\alpha_j={\rm const}$, $i,j=1,2$. In that case $\frac{\varphi_1}{\varphi_2}=\frac{c_1}{c_2}={\rm const}$, and it is visible that such $\varphi_i$ do not satisfy equations
\eqref{wz1}--\eqref{wz3}.

If we take $a_8=-\beta$, then
\[Z=\ln c -\frac{2a_7}{\beta T}\]
but most importantly the integral in \eqref{integral} can be easily calculated,
\[W=-\frac{1}{\beta c a_7}{\rm e}^{\frac{2a_7}{\beta T}}+c_0.\]
From this we obtain
\[-\frac{2a_7}{\beta T}=\ln \left(\frac{1}{\beta a_7 c\,(c_0-W)} \right).\]
Taking the equations above into account, we see that
\[{\rm e}^Z=\frac{1}{\beta a_7 (c_0-W)},\quad T=\frac{2a_7}{\beta\ln(\beta a_7 c\,(c_0-W))}.\]
Because $W=t$, $\varphi_2={\rm e}^Z$, $\varphi_1=T\varphi_2$, the solution of the reduced system in original coordinates is
\[\varphi_2=\frac{1}{\beta a_7 (c_0-t)},\quad \varphi_1=\frac{2}{\beta^2 (c_0-t)\ln(\beta a_7 c\, (c_0-t))}.\]
Solutions we constructed have this property that corresponding solutions $\varphi_1$, $\varphi_2$ are not zeros (identically). The least we can say that there exist $a_8$ such that the proposed method allows us to construct a solution which is not invariant (can't be obtained by classical Lie method).

In conclusion, we have found $\varphi_i$'s for which the solution $u(x,t)=u(x,\varphi_1(t),\varphi_2(t))$ is not invariant under the point symmetries of the evolution equation~\eqref{a6a7a8}.

Let's take equation
\begin{equation}
u_t=\left(\frac{{\rm e}^{\beta x}}{u}\right)_{xx}.
\label{gole}
\end{equation}
The basis of its Lie algebra of Lie point symmetries is
\begin{equation}
X_1=\partial_t, \quad X_2=2t\partial_t+u\partial_u,\quad X_3=2\partial_x+\beta u\partial_u.
\label{genera}
\end{equation}
The solution to reduction equation
\begin{equation}\label{red01}
\varphi_1'+\frac{1}{2}\beta^2\varphi_1^2=0,\quad
\varphi_2'+\beta^2\varphi_1\varphi_2=0
\end{equation}
is
\begin{equation}
\varphi_1=\frac{2}{\beta^2t+c_1},\quad \varphi_2=\frac{c_2}{(\beta^2t+c_1)^2},\quad c_1, c_2={\rm const}.
\label{r2}
\end{equation}
We call ansatz \eqref{ansatzE} together with solutions \eqref{r2} a partial solution of \eqref{gole}.

\begin{twr}Let $A_3$ be a Lie algebra of Lie point symmetries of \eqref{gole} with basis \eqref{genera}. A partial solution \eqref{ansatzE}, \eqref{r2} is invariant with respect to one-parameter Lie group with generator $\alpha_1X_1+\alpha_2X_2+\alpha_3X_3$, where $\alpha_1$, $\alpha_2$, $\alpha_3$ depend on $c_1$, $c_2$.
\end{twr}
\begin{proof}
Note that in general case $\frac{\partial u}{\partial \varphi_1}$ and $\frac{\partial u}{\partial \varphi_2}$ are linearly independent. Otherwise, for some $\beta_1$, $\beta_2$, equation $\beta_1\frac{\partial u}{\partial \varphi_1}+\beta_2\frac{\partial u}{\partial \varphi_2}=0$ would be true, and
therefore $u$ would be dependent on only one constant,
\[u=f(x,t,\beta_2\varphi_1-\beta_1\varphi_2),\]
which is impossible, because it is a general solution of a second order ordinary differential equation~\eqref{zwy3}.
Moreover $\frac{\partial u}{\partial c_1}$, $\frac{\partial u}{\partial c_2}$ are linearly independent since $\varphi_1$,
$\varphi_2$ is the general solution of \eqref{red01} and has the form~\eqref{r2}.

The action of the operator $X=\xi_j(x,u)\frac{\partial}{\partial x_j}+\eta(x,u)\frac{\partial}{\partial u}$, is the following
\[X\big(h(x)-u)\big)\big|_{u=h(x)}=\left(\xi_j(x,u)\frac{\partial h(x)}{\partial x_j}-\eta(x,u)\right)\Big|_{u=h(x)},\]
where $x=(x_1,\ldots,x_n)$ for some integer $n$, $h(x)$ is a differentiable function. Then we show that
\[Q\big(f(x,t)-u)\big)\big|_{u=f(x,t)}\in W_2\]
where $W_2=\operatorname{span}\big\{\frac{\partial u}{\partial c_1}, \frac{\partial u}{\partial c_2}\big\}$ and $Q\in A_3$.
To prove this, it is enough to show this property for each of the basis elements $X_1$, $X_2$, $X_3$.

By direct computation we show that for $X_1=\partial_t$
\[X_1\big(f(x,t)-u)\big)\big|_{u=f(x,t)}=\beta^2\frac{\partial u}{\partial c_1},\]
for $X_2=2t\partial_t+u\partial_u$
\[X_2\big(f(x,t)-u)\big)\big|_{u=f(x,t)}=-2c_1\frac{\partial u}{\partial c_1}-2c_2\frac{\partial u}{\partial c_2}\]
and for $X_3=2\partial_x+\beta u\partial_u$
\[X_3\big(f(x,t)-u)\big)\big|_{u=f(x,t)}=-2\beta c_2\frac{\partial u}{\partial c_2},\]
where $f(x,t)$ is the solution of \eqref{gole} given by \eqref{ansatzE}, \eqref{r2}.

From the fact that any three vectors in two-dimensional vector space are linearly dependent, it follows that for any special solution \eqref{ansatzE}, \eqref{r2} there can be selected $\alpha_1$, $\alpha_2$, $\alpha_3$ such that
$X \big(f(x,t)-u)\big)\big|_{u=f(x,t)}=0$, where $X=\alpha_1X_1+\alpha_2X_2+\alpha_3X_3$ and not all $\alpha_i$ are equal to zero.\end{proof}
We conclude that every solution given by \eqref{ansatzE}, \eqref{r2} can be found using a classical method of invariant solutions with respect to a one-parameter Lie group $\alpha_1X_1+\alpha_2X_2+\alpha_3X_3$.

This theorem is only sufficient but not necessary condition for a solution to be invariant in a classical sense. In fact, let us consider the 2-dimensional abelian subalgebra of Lie algebra $A_3$, with basis elements $Q_1$, $Q_2$, where
\begin{gather*}Q_1=2c_1X_1+\beta^2X_2=2\big(\beta^2t+c_1\big)\partial_t+\beta^2u\partial_u,\\
 Q_2=X_3=2\partial_x+\beta u\partial_u\end{gather*}
(these operators clearly commute, $\left[Q_1,Q_2\right]=0$). Because
\[Q_1\big(f(x,t)-u)\big)\big|_{u=f(x,t)}=-2c_2\beta^2\frac{\partial u}{\partial c_2}\]
and
\[Q_2\big(f(x,t)-u)\big)\big|_{u=f(x,t)}=-2\beta c_2\frac{\partial u}{\partial c_2},\]
the solution $u=f(x,t)$ is invariant with respect to a linear combination $Q_1-\beta Q_2$.
On the other hand, it is obvious that this solution cannot be obtained with a classical method using just any 2-dimensional subalgebra, because not every 2-dimensional algebra has the aforementioned properties, for example $\{X_1,X_3\}$. No nonzero linear combination of $X_1$ and $X_3$ leaves the solution invariant.

\section{Conclusions}

We obtain the large classes of evolution and non-evolution type equations to which the symmetry reduction method proposed in~\cite{Tsy} is applicable. We construct such equations by using the generalized symmetry of second order ODEs. We obtain the solutions of nonlinear hyperbolic equations \eqref{waveeqat} and \eqref{waveeqa1}, which depend on a arbitrary function $\varphi_1(x_2)$.

We show that in the case when nonlinear heat equations invariant with respect to one- and two-parameter groups one can
find solutions which don't satisfy the invariance surface condition and thus can not be obtain by virtue of classical infinitesimal method. It was shown in Section 2 that the approach is also applicable to ordinary differential equations.
Note that the reduction method can also be applied for reducing the Cauchy problem for PDEs to the Cauchy problem for a system of ODEs.

\end{document}